\documentclass[12pt]{amsart}

\title[Upper bound for determinant Waring rank]{An improved upper bound for the Waring rank of the determinant}
\date{\today}

\author{Garritt Johns}
\email{garrittjohns@u.boisestate.edu}
\address{Department of Mathematics \\
Boise State University \\
1910 University Drive \\
Boise, ID \ 83725-1555 \\
USA}

\author{Zach Teitler}
\email{zteitler@boisestate.edu}
\address{Department of Mathematics \\
Boise State University \\
1910 University Drive \\
Boise, ID \ 83725-1555 \\
USA}

\usepackage{amsmath,amsfonts,amssymb,amsthm}
\usepackage[margin=2.5cm]{geometry}

\PassOptionsToPackage{obeyspaces}{url}
\usepackage[pagebackref]{hyperref}

\usepackage{booktabs}

\newtheorem{theorem}{Theorem}
\newtheorem{lemma}[theorem]{Lemma}

\theoremstyle{definition}
\newtheorem{definition}[theorem]{Definition}

\theoremstyle{remark}
\newtheorem{remark}[theorem]{Remark}

\newcommand{\bbC}{\mathbb{C}}
\newcommand{\bbK}{\mathbb{K}}
\newcommand{\bbP}{\mathbb{P}}
\newcommand{\bbZ}{\mathbb{Z}}
\newcommand{\field}{\Bbbk}

\newcommand{\dett}{\det\nolimits}

\DeclareMathOperator{\rank}{rank}
\DeclareMathOperator{\Sym}{Sym}

\DeclareMathOperator{\Supp}{Supp}
\DeclareMathOperator{\diag}{diag}
\DeclareMathOperator{\Aff}{Aff}

\newcommand{\defining}[1]{\textbf{#1}}

\makeatletter
\@namedef{subjclassname@2020}{\textup{2020} Mathematics Subject Classification}
\makeatother

\begin{document}

\begin{abstract}
The Waring rank of the generic $d \times d$ determinant is bounded above by $d \cdot d!$.
This improves previous upper bounds, which were of the form an exponential times the factorial.
Our upper bound comes from an explicit power sum decomposition.
We describe some of the symmetries of the decomposition and set-theoretic defining equations for the terms of the decomposition.
\end{abstract}

\maketitle

For a homogeneous polynomial $F$ of degree $d$,
the \defining{Waring rank} of $F$, denoted $\rank(F)$, is the least integer $r$
such that $F$ can be expressed as a linear combination of $r$ terms
which are each $d$th powers of linear forms.
For example,
\[
  xy = \frac{1}{4} \Big( (x+y)^2 - (x-y)^2 \Big),
\]
so $\rank(xy) \leq 2$.
Similarly
\[
  xyz = \frac{1}{24} \Big( (x+y+z)^3 - (x+y-z)^3 - (x-y+z)^3 + (x-y-z)^3 \Big),
\]
so $\rank(xyz) \leq 4$.

Let $\dett_d$ denote the generic $d \times d$ determinant,
that is, the determinant of the $d \times d$ matrix $(x_{i,j})_{1 \leq i,j \leq d}$
whose entries are independent variables.
Previous upper bounds for $\rank(\dett_d)$ were $2^{d-1} \cdot d!$,
later improved to $\left(\frac{5}{6}\right)^{\lfloor d/3 \rfloor} 2^{d-1} \cdot d!$;
these have the form, an \emph{exponential} function times a factorial.
We give an explicit expression to show a new upper bound for the Waring rank of the determinant,
which is a \emph{linear} function times the factorial, namely,
\begin{equation}
\label{equation: rank bound}
  \rank(\dett_d) \leq d \cdot d! .
\end{equation}
This holds over the complex numbers,
or more generally over any field or commutative ring where $d!$ is invertible and there is a primitive $d$th root of unity.%
\footnote{J.M.~Landsberg informed us of
an unpublished result of Gurvits giving an
upper bound of $(d+1) \cdot d!$, see below.}

Specifically, we show that
\begin{equation}
\label{equation: explicit power sum decomposition}
  d \cdot d! \dett_d
    = \sum_{\sigma \in S_d} (-1)^{\sigma} \sum_{j=1}^d
      (-1)^{(d+1)j} \left( \sum_{i=1}^d \omega^{ij} x_{i,\sigma i} \right)^d,
\end{equation}
where $\omega$ is a primitive $d$th root of unity and $S_d$ denotes the symmetric group on $d$ letters.
This is a linear combination of $d \cdot d!$ terms which are $d$th powers of linear forms,
with coefficients $\pm 1$.

In addition, we describe some of the symmetries of the decomposition, and set-theoretic defining equations for the terms of the decomposition.

\section{Background}

Fix a degree $d$.
Let $\field$ be a field or commutative ring in which $d!$ is invertible;
in particular, the characteristic is zero or greater than $d$.
We consider homogeneous polynomials $F \in \field[x_1,\dotsc,x_n]$ of degree $d$.

We denote by $\rank_{\field}(F)$, or simply $\rank(F)$, the least number $r$ of linear forms $\ell_1,\dotsc,\ell_r \in \field[x_1,\dotsc,x_n]$ such that $F = c_1 \ell_1^d + \dotsb + c_r \ell_r^d$ for some $c_1,\dotsc,c_r \in \field$.
(The term ``Waring rank'' is often reserved for the case that $\field$ is a field, even an algebraically closed field.)
If $\field \subseteq \bbK$ is an extension field or ring,
then evidently $\rank_{\field}(F) \geq \rank_{\bbK}(F)$.
Thus, upper bounds for $\rank_{\field}(F)$ are also upper bounds for $\rank_{\bbK}(F)$.
For this reason we will describe our upper bound for $\rank(\dett_d)$ over the smallest $\field$ possible.
First, in this section, we describe previously known upper and lower bounds.
All the results described in this section are valid at least for $\field=\bbC$; we make some partial indications of more general $\field$ where they hold.

We have
\begin{equation}
\label{eq: power sum decomposition of product}
  2^{d-1} \cdot d! \; x_1 \dotsm x_d
    = \sum_{\substack{\epsilon \in \{\pm 1\}^d \\ \epsilon_1 = +1}}
      \left( \prod_{i=1}^d \epsilon_i \right) (\epsilon_1 x_1 + \dotsb + \epsilon_d x_d)^d,
\end{equation}
which shows that $\rank(x_1 \dotsm x_d) \leq 2^{d-1} < \infty$
whenever $2^{d-1} \cdot d!$ is a unit in $\field$.
By substitution, any monomial of degree $d$ has finite rank, and then so does any homogeneous form,
simply by decomposing each monomial into a sum of powers.
This shows that every homogeneous form of degree $d$ has finite rank.

One can show that in fact $\rank(xy) = 2$ and $\rank(xyz) = 4$, and more generally
\[
  \rank(x_1 \dotsm x_d) = 2^{d-1},
\]
see \cite{MR2842085}.

The determinant $\dett_d$ is a (signed) sum of $d!$ monomials, each of which is of the form $x_1 \dotsm x_d$.
For example, $\dett_3 = x_{1,1} x_{2,2} x_{3,3} - \dotsb$, a sum of $6$ terms which each have the form $xyz$.
Therefore $\dett_3$ can be written as a sum of $6 \cdot \rank(xyz) = 24$ powers of linear forms.

In general,
\[
  \dett_d = \sum_{\sigma \in S_d} (-1)^{\sigma} x_{1,\sigma 1} \dotsm x_{d,\sigma d}.
\]
Combining with \eqref{eq: power sum decomposition of product}, this yields the ``classical'' power sum decomposition
\begin{equation}\label{equation: classical decomposition}
  2^{d-1} \cdot d! \dett_d = \sum_{\sigma \in S_d} (-1)^{\sigma}
    \sum_{\substack{\epsilon \in \{\pm1\}^d \\ \epsilon_1 = +1}} \left( \prod_{i=1}^d \epsilon_i \right)
      (\epsilon_1 x_{1,\sigma 1} + \dotsb + \epsilon_d x_{d,\sigma d})^d,
\end{equation}
with $2^{d-1} \cdot d!$ terms.
That is,
\begin{equation}\label{equation: classical upper bound}
    \rank(\dett_d) \leq 2^{d-1} \cdot d! .
\end{equation}
So $\rank(\dett_d)$ for $d = 3,4,5,6,\dotsc$ are bounded above by $24$, $192$, $1920$, $23040$, and so on.

Derksen \cite{MR3494510} and, later, Krishna-Makam \cite{Krishna:2018aa} found expressions for $\dett_3$ as a sum of $5$ terms
which are products of linear forms.
The identity of Krishna-Makam, which is slightly simpler, is as follows:
\[ % Krishna-Makam's identity
\begin{split}
  \det\nolimits_3 &= x_{1,1} \, (x_{2,2}+x_{2,3}) \, (x_{3,1}+x_{3,3}) \\
    & \quad + (x_{1,2} + x_{1,3}) \, x_{2,1} \, x_{3,2} \\
    & \quad - (x_{1,1} + x_{1,3}) \, x_{2,2} \, x_{3,1} \\
    & \quad - x_{1,2} \, (x_{2,1} + x_{2,3}) \, (x_{3,2} + x_{3,3}) \\
    & \quad + (x_{1,2} - x_{1,1}) \, x_{2,3} \, (x_{3,1} + x_{3,2} + x_{3,3}).
\end{split}
\]
Notably, this holds over the integers (and even in characteristic $2$).
The Derksen and Krishna-Makam identities give $\rank(\dett_3) \leq 5 \rank(xyz) = 20$.
Derksen observed that by Laplace expansion,
\begin{equation}
\label{equation: derksen upper bound}
  \rank(\dett_d) \leq \left( \frac{5}{6} \right)^{\lfloor d/3 \rfloor} 2^{d-1} \cdot d! .
\end{equation}
In particular $\rank(\dett_d)$ for $d=3,4,5,6,\dotsc$ are bounded above by $20$, $160$, $1600$, $16000$, and so on.

Conner-Gesmundo-Landsberg-Ventura gave an explicit expression for $\dett_3$
as a sum of $18$ cubes of linear forms over $\bbC$ \cite[Theorem 2.11]{Conner:2019aa},
showing
\begin{equation}
    \label{equation: CGLV bound}
    \rank_{\bbC}(\dett_3) \leq 18 .
\end{equation}
Our expression \eqref{equation: explicit power sum decomposition} is a direct generalization of theirs;
their expression is exactly the case $d=3$ of ours.

J.M.~Landsberg informed us of an unpublished result of Gurvits,
that $\rank_{\field}(\dett_d) \leq (d+1) \cdot d!$.
This follows from the identity
\begin{equation}
\label{equation: gurvits decomposition}
  d! \cdot \dett_d = \sum_{\sigma \in S_d} (-1)^{\sigma}
    \left( \left(\sum_{i=1}^d x_{i,\sigma i}\right)^d
      - \sum_{j=1}^d \left( \sum_{\substack{1 \leq i \leq d \\ i \neq j}} x_{i,\sigma i} \right)^d \right).
\end{equation}
This decomposition is valid over the integers, so the upper bound holds
whenever $d!$ is a unit in $\field$.

Lower bounds for $\rank(\dett_d)$ have been studied by several authors, including
\cite{MR2628829}, \cite{MR3316987}, \cite{MR3390032}, \cite{MR3427655}, \cite{Boij:aa}.
Currently, the best lower bounds, for algebraically closed fields $\field$, are as follows.
For all $d \geq 3$, $\rank(\dett_d) \geq \binom{2d}{d} - \binom{2d-2}{d-1}$:
thus for $d=3,4,5,6,\dotsc$, $\rank(\dett_d)$ is bounded below by $14$, $50$, $182$, $672$, and so on, see \cite{MR3390032}.
Recently, \cite{Boij:aa} improved this to $\rank(\dett_3) \geq 15$.
Very soon after that, \cite{Conner:2019ab} improved it further to $\rank_{\bbC}(\dett_3) \geq 17$.

We summarize the previous and new bounds in Table~\ref{table: bounds}.

\begin{table}
\[
\begin{array}{r rrrr rrrr} %r}
\toprule
d & 2 & 3 & 4 & 5 & 6 & 7 & 8 & 9 \\ %& 10 \\
\midrule
\text{classical upper bound \eqref{equation: classical upper bound}}
  & 4 & 24 & 192 & 1920 & 23040 & 322560 & 5160960 & 92897280 \\ %& 1857945600 \\
\text{Derksen upper bound \eqref{equation: derksen upper bound}}
  & 4 & 20 & 160 & 1600 & 16000 & 224000 & 3584000 & 53760000 \\ %& 1075200000 \\
\text{Gurvits upper bound \eqref{equation: gurvits decomposition}}
  & 6 & 24 & 120 & 720 & 5040 & 40320 & 362880 & 3628800 \\% & 39916800 \\
\text{CGLV upper bound \eqref{equation: CGLV bound}}
  & & 18 & & & & & & \\ %& \\
\text{new upper bound \eqref{equation: rank bound}}
  & 4 & 18 &  96 &  600 &  4320 &  35280 &  322560 &  3265920 \\ %&   36288000 \\
\midrule
\text{lower bound}          & 4 & 17 &  50 &  182 &   672 &   2508 &    9438 &    35750 \\ %&     136136 \\
\bottomrule
\end{array}
\]
\caption{Summary of bounds for $\rank_{\bbC}(\dett_d)$.}
\label{table: bounds}
\end{table}

\subsection{Notation}

We use multi-index notation as follows.
For a tuple $\alpha = (\alpha_1,\dotsc,\alpha_n)$,
$x^\alpha$ denotes $x_1^{\alpha_1} \dotsm x_n^{\alpha_n}$.
We write $|\alpha| = \sum \alpha_i$, $\alpha! = \prod \alpha_i!$,
and $\binom{d}{\alpha} = \frac{d!}{\alpha!}$ when $d=|\alpha|$.

For a pure (coefficient $1$) monomial $M$ and polynomial $P$, $[M]P$ denotes the coefficient of $M$ in $P$.
For example, $[xy](x+y)^2 = 2$.
More generally, $[x^\alpha](x_1+\dotsb+x_n)^{|\alpha|} = \binom{|\alpha|}{\alpha}$.

$S_d$ denotes the symmetric group on $[d]=\{1,\dotsc,d\}$.
For $\sigma \in S_d$ and $i \in [d]$ we write $\sigma i$ for the result of the permutation $\sigma$ applied to $i$.
We write $(-1)^\sigma$ for the sign of the permutation $\sigma$.

\section{Proofs}

\begin{theorem}\label{thrm: decomposition of det into powers of linear forms}
Let $\field$ be any commutative ring containing a primitive $d$th root of unity $\omega$.
Then over $\field$, \eqref{equation: explicit power sum decomposition} holds,
\[
  d \cdot d! \dett_d
    = \sum_{\sigma \in S_d} (-1)^{\sigma} \sum_{j=1}^d
      (-1)^{(d+1)j} \left( \sum_{i=1}^d \omega^{ij} x_{i,\sigma i} \right)^d.
\]
In particular, if also $d!$ is invertible in $\field$, then $\rank(\dett_d) \leq d \cdot d!$.
\end{theorem}
We begin by proving the following lemma.

\begin{lemma}\label{lem: expanding via multinomial theorem}
Fix a degree $d$ monomial $x_{i_1} \dotsm x_{i_d}$,
where the indices $1 \leq i_1 \leq \dotsb \leq i_d \leq d$ may repeat.
Let $\lambda = (\lambda_1,\dotsc,\lambda_d)$ be the $d$-tuple of multiplicities
of elements in the multiset $I = \{i_1,\dotsc,i_d\}$.
That is, each $\lambda_k$ is the number of $j$ such that $i_j = k$.
Then
\begin{equation}
[x_{i_1} \dotsm x_{i_d}]\sum_{j=1}^{d} (-1)^{(d+1)j} \left(\sum_{i=1}^{d}\omega^{ij}x_i\right)^d = 
\begin{cases}
\binom{d}{\lambda} d & \text{if } \sum_{k=1}^d i_k \equiv \binom{d+1}{2} \pmod{d} , \\
0 & \text{otherwise.}
\end{cases}
\end{equation}
\end{lemma}
\begin{proof}%[Proof of Lemma \ref{lem: expanding via multinomial theorem}]
We write $x_I = x_{i_1} x_{i_2} \dotsm x_{i_d} = x_1^{\lambda_1} x_2^{\lambda_2} \dotsm x_d^{\lambda_d}$.
Let
\begin{equation*}
P=\sum_{j=1}^{d} (-1)^{(d+1)j} \left(\sum_{i=1}^{d}\omega^{ij}x_i\right)^d .
\end{equation*}
We aim to find the coefficient of $x_I$ in $P$:
\begin{equation*}%\label{equation: coefficient of x_I in P}
[x_I]P = \sum_{j=1}^d (-1)^{(d+1)j} \cdot [x_I]\left(\sum_{i=1}^{d}\omega^{ij}x_i\right)^d.
\end{equation*}
After applying the multinomial theorem, this becomes
%\eqref{equation: coefficient of x_I in P} becomes
\begin{equation*}
[x_I]P = \sum_{j=1}^{d} (-1)^{(d+1)j} \cdot \binom{d}{\lambda} \prod_{k=1}^{d} \omega^{i_k j} 
  = \binom{d}{\lambda} \sum_{j=1}^d \left( (-1)^{d+1} \omega^{\sum i_k} \right)^j .
\end{equation*}
Now $-1 = \omega^{d/2}$, so $(-1)^{d+1} = (-1)^{-(d+1)} = \omega^{-\binom{d+1}{2}}$.
Thus
\begin{equation*}
[x_I]P = \binom{d}{\lambda} \sum_{j=1}^d \left( \omega^{\sum i_k - \binom{d+1}{2}} \right)^j .
\end{equation*}
The claim is proved when we
recall that $\sum_{j=1}^d \omega^{pj}$ is $d$ if $p \equiv 0 \pmod{d}$, $0$ otherwise.
\end{proof}

Note that if $d$ is even, then $\binom{d+1}{2} \equiv d/2 \pmod{d}$, while if $d$ is odd, then $\binom{d+1}{2} \equiv 0 \pmod{d}$.
Either way, $\binom{d+1}{2} \equiv - \binom{d+1}{2} \pmod{d}$.

Now we proceed with the proof of Theorem \ref{thrm: decomposition of det into powers of linear forms}.

\begin{proof}[Proof of Theorem \ref{thrm: decomposition of det into powers of linear forms}]

Let 
\begin{equation*}
R=\sum_{\sigma\in S_d}(-1)^{\sigma}\sum_{j=1}^{d} (-1)^{(d+1)j} \left(\sum_{i=1}^{d}\omega^{ij}x_{i, \sigma i}\right)^d.
\end{equation*}
The claim is that $R = d \cdot d! \dett_d$.

Let $I = (i_1,i_2,\dots,i_d)$ and $J = (j_1,j_2,\dots,j_d)$ be $d$-tuples of elements in $ [d] $.
We denote by $x_{I,J}$ the monomial $x_{I,J} = x_{i_1,j_1} \dotsm x_{i_d,j_d}$.

The coefficient of $x_{I,J}$ in $\dett_d$ is easy to find.
Denote by $\Supp(I)$ the support of $I$, that is, the subset of $[d]$ of values that appear in $I$;
and similarly $\Supp(J)$.
If $\Supp(I) = \Supp(J) = [d]$, then there is a unique permutation $\sigma_{I,J} \in S_d$
such that $\sigma_{I,J} i_k = j_k$ for $k=1,\dotsc,d$.
In this case $[x_{I,J}] \dett_d = (-1)^{\sigma_{I,J}}$.
Otherwise, if $\Supp(I) \neq [d]$ or $\Supp(J) \neq [d]$, then $[x_{I,J}] \dett_d = 0$.

Now we consider the coefficient of $x_{I,J}$ in $R$.
%Let $\bar{I}$ be the (simple) set $[d]\setminus \Supp(I)$.
Let $H = \{\sigma\in S_d \mid \sigma i_1 = j_1, \sigma i_2 = j_2, \dots, \sigma i_d = j_d\}$.
We have
\[
  [x_{I,J}] R = \sum_{\sigma \in H} (-1)^\sigma [x_{I,J}] \sum_{j=1}^d (-1)^{(d+1)j} \left( \sum_{i=1}^d \omega^{ij} x_{i,\sigma i} \right)^d ,
\]
since $x_{I,J}$ simply involves variables that do not appear in the $\sigma$ terms for $\sigma \notin H$.
In particular, if $H = \varnothing$ then $[x_{I,J}]R = 0$.
%Indeed, if $[x_{I,J}]R \neq 0$ then there exists a $\sigma\in S_d$ such that
%$x_{I,J}$ arises as a term in the power of a linear form involving variables $x_{1,\sigma 1}, \dotsc, x_{d,\sigma d}$.
%That means $\{x_{i_1,j_1},\dotsc, x_{i_d,j_d}\} \subseteq \{x_{1, \sigma 1}, \dotsc , x_{d, \sigma d}\}$.
%Therefore $\sigma i_k = j_k$ for $k = 1,\dots,d$, so $\sigma \in H$ and $ H \neq \varnothing $.
(We will not use this fact, but the case $H = \varnothing$ occurs
when there are $k,k'$ such that $i_k = i_{k'}$ but $j_k \neq j_{k'}$,
or $j_k = j_{k'}$ but $i_k \neq i_{k'}$.)

Now suppose $H \neq \varnothing$.
For $\sigma \in H$ we have $x_{I,J} = x_{i_1, \sigma i_1} \dotsm x_{i_d, \sigma i_d} = x_{1,\sigma 1}^{\lambda_1} \dotsm x_{d,\sigma d}^{\lambda d}$.
By Lemma~\ref{lem: expanding via multinomial theorem}, applied to the variables $x_k = x_{k,\sigma k}$,
we have
\[
  [x_{I,J}] \sum_{j=1}^d (-1)^{(d+1)j} \left( \sum_{i=1}^d \omega^{ij} x_{i,\sigma i} \right)^d =
  \begin{cases}
    \binom{d}{\lambda} d, & \text{if } \sum_{k=1}^d i_k \equiv \binom{d+1}{2} \pmod{d}, \\
    0, & \text{otherwise} .
  \end{cases}
\]
Therefore
\[
  [x_{I,J}]R = \begin{cases}
    \binom{d}{\lambda} d \sum_{\sigma \in H} (-1)^\sigma, & \text{if } \sum_{k=1}^d i_k \equiv \binom{d+1}{2} \pmod{d}, \\
    0, & \text{otherwise} .
  \end{cases}
\]

Let $ \sigma_0 \in H $ and let $ H_0 = \sigma_0^{-1} H $.
%Let $\bar{I}$ be the (simple) set $[d]\setminus \Supp(I)$.
Observe $ H_0 $ is precisely the subgroup of $S_d$ consisting of elements that fix $\Supp(I)$ pointwise,
so $H_0$ is (isomorphic to) the symmetric group on $[d] \setminus \Supp(I)$.
Indeed, $\sigma \in H_0$ if and only if $\sigma_0 \sigma \in H$,
if and only if $\sigma_0 \sigma i_k = j_k = \sigma_0 i_k$ for all $k$,
if and only if $\sigma i_k = i_k$ for all $k$.
That is, $H_0 = \{\sigma \in S_d \mid \sigma i_1 = i_1, \dots, \sigma i_d = i_d \}$.
%So we see that $ \varphi: H_0 \to \Sym(\bar{I}) $ by $ \sigma \mapsto \sigma|_{\bar{I}} $ is an isomorphism.

%Now
%\begin{equation}
%\begin{split}
%[x_{I,J}]R &= \sum_{\sigma \in S_d} (-1)^{\sigma} \cdot  [x_{I,J}]\sum_{j=1}^{d} (-1)^{(d+1)j} \left( \sum_{i=1}^{d} \omega^{ij}x_{i,\sigma i} \right)^d \\
%&= 
%\begin{cases}
% (-1)^{\sigma_0} \sum_{ \sigma \in H_0}(-1)^{\sigma} \cdot d \binom{d}{\lambda} & \text{if } d\mid \sum_{k=1}^{d} k + i_k,\\
% 0 & \text{if } d\nmid \sum_{k=1}^{d} k + i_k.
%\end{cases}
%\end{split}
%\end{equation}

%If $|H|\geq 2$ then there exist $\sigma_1,\sigma_2 \in H$ with $\sigma_1 \neq \sigma_2$. 
%Therefore there exists an integer $a \in [d]$ such that $\sigma_1 a \neq \sigma_2 a$.
%If $a \in I$ then for some $k\in [d]$ we have $\sigma_1 a = j_k = \sigma_2 a$, a contradiction.
%Thus $a \not\in I$.
%%The above two lines might be a bit too obvious to spend time on. (GJ)
%Thus $I$ has at least one element with multiplicity greater than one, say $i_m = i_n$.
%Let $\tau = (m \; n) \in S_d$ transpose $m$ and $n$.
%Note that if $\sigma \in H$ then so is $\tau \sigma$.
%Therefore $\sum_{\sigma \in H} (-1)^{\sigma} = 0$ and hence $[x_{I,J}]R = 0$ as well.

Now if $ |H| \geq 2 $ then
%$ \sum_{ \sigma \in H_0}(-1)^{\sigma} = 0 $ and hence $ [x_{IJ}]R = 0 $.
the subgroup $H_0$ consists of an equal number of even and odd permutations,
and so does its coset $H$.
Therefore $\sum_{\sigma H} (-1)^\sigma = 0$, so $[x_{I,J}]R = 0$.

If $|H| = 1$ then
%$ |\bar{I}| = 0 $ or $ 1 $.
$\Supp(I)$ consists of either $d-1$ or $d$ elements of $[d]$.

Suppose $|\Supp(I)| = d-1$, so (as multisets)
$\{i_1,\dotsc,i_d\} = ([d] \setminus \{m\}) \cup \{n\}$ for some $m,n \in [d]$, $m \neq n$.
In this case
\begin{equation*}
     \sum_{k=1}^{d} i_k = (n - m) + \sum_{k=1}^{d} k = (n - m) + \binom{d+1}{2}.
\end{equation*}
Now $|n - m| < d$, hence $d\nmid n-m$ and so
$\sum_{k=1}^{d} i_k \not \equiv \binom{d+1}{2} \pmod{d}$.
Thus $[x_{I,J}]R = 0$.

Finally, we suppose $\Supp(I) = [d]$.
Note that when $\{i_1,\dots,i_k\} = [d] $ then $|H| = 1$ if and only if $\{j_1,\dots,j_d\} = [d]$ as well.
The unique permutation $\sigma_0 \in H$ must be $\sigma_0 = \sigma_{I,J}$, the permutation given by
$\sigma i_k = j_k$ for each $k$.
We have $\sum_{k=1}^d i_k = \sum_{k=1}^d k = \binom{d+1}{2}$.
Thus $[x_{I,J}]R = (-1)^{\sigma_{I,J}} \cdot d \binom{d}{\lambda}$.
In this case $\lambda = (1,\dots,1)$, and $\binom{d}{\lambda} = d!$.
Therefore $[x_{I,J}] R = (-1)^{\sigma_{I,J}} d \cdot d!$.

This proves
\begin{equation}
 [x_{I,J}]R = 
 \begin{cases}
     (-1)^{\sigma_{I,J}} d \cdot d! & \text{if } \Supp(I) = \Supp(J) = [d], \\
     0 & \text{otherwise}.
 \end{cases}
\end{equation}
%Thus
%\begin{equation}
%\begin{split}
% 	R &= \sum [x_{IJ}]R \cdot \prod_{k=1}^{d} x_{i_k j_k}\\ 
% 	&= d \cdot d! \sum_{\sigma \in S_d} (-1)^\sigma \prod_{i=1}^{d} x_{i,\sigma i} \\
% 	&= d \cdot d! \dett_d
%\end{split}
%\end{equation}
This is $d \cdot d!$ times the coefficient $[x_{I,J}] \dett_d$, for all $I,J$.
Therefore $R = d \cdot d! \dett_d$, as claimed.
\end{proof}

\section{Linear independence}

In this section $\field$ is a field.

For $\sigma \in S_d$ and $j = 1,\dotsc,d$,
let $T_{\sigma,j} = (-1)^{\sigma}(-1)^{(d+1)j}(\omega^j x_{1,\sigma 1} + \dotsb + \omega^{dj} x_{d,\sigma d})^d$,
so Theorem~\ref{thrm: decomposition of det into powers of linear forms} is that
\[
%   d \cdot d! \dett_d = \sum_{\sigma \in S_d} (-1)^{\sigma} \sum_{j=1}^d (-1)^{(d+1)^j}
  d \cdot d! \dett_d = \sum_{\sigma,j}T_{\sigma,j} .
\]
\begin{theorem}
The terms $T_{\sigma,j}$ appearing in \eqref{equation: explicit power sum decomposition} are linearly independent.
\end{theorem}
\begin{proof}
We show that there are linear functionals $\ell_{\sigma,j}$ on the space of homogeneous forms of degree $d$
such that $\ell_{\sigma,j}(T_{\psi,k})$ is nonzero if $\psi=\sigma$, $k=j$, and zero otherwise.
Recall that linear functionals correspond to ``dual'' homogeneous forms of degree $d$
in such a way that evaluating such a functional on a $d$th power $\ell^d$ of a linear form $\ell$
corresponds to evaluating the dual form at the point whose coordinates are the coefficients of $\ell$.
Let $P_{\sigma,j} \in \field^{d^2}$ be the point whose coordinates are the coefficients of
$\sum_{i=1}^d \omega^{ij} x_{i,\sigma i}$,
so the $x_{i,k}$ coefficient of $P_{\sigma,j}$ is $\omega^{ij}$ if $k = \sigma i$, $0$ otherwise.
Thus it is sufficient to show that, for each $\sigma,j$, there is a degree $d$ form
which is nonvanishing at the point $P_{\sigma,j}$
and vanishing at all of the points $P_{\psi,k}$ for $\psi \neq \sigma$ or $k \neq j$.

In fact we explicitly produce such forms of degree $d-1$.
The appropriate forms of degree $d$ can be obtained by multiplying by appropriate linear forms
(to increment the degree); any linear form nonvanishing at $P_{\sigma,j}$ will do.

Given $\sigma \in S_d$, for each $k$ let
$e_{\sigma,k} = x_{1,\sigma 1} \dotsm \widehat{x_{k,\sigma k}} \dotsm x_{d,\sigma d}$.
Now given $j$, let $L_{\sigma,j}$ be the degree $d-1$ form
\[
  L_{\sigma,j} = \sum_{k=1}^d \omega^{kj} e_{\sigma,k} .
\]
We evaluate $L_{\sigma,j}$ first at the point $P_{\sigma,j}$.
Observe
\[
  e_{\sigma,k}(P_{\sigma,j}) = \omega^{1j} \dotsm \widehat{\omega^{kj}} \dotsm \omega^{dj} = \omega^{\binom{d+1}{2} j - kj} .
\]
Thus
\[
  L_{\sigma,j}(P_{\sigma,j}) = \sum_{k=1}^d \omega^{\binom{d+1}{2} j} = (-1)^{(d+1)j} d ,
\]
which is indeed nonzero.

For $m \neq j$,
\[
  L_{\sigma,j}(P_{\sigma,m}) = \sum_{k=1}^d \omega^{kj} \omega^{\binom{d+1}{2} m - km}
    = (-1)^{(d+1)m} \sum_{k=1}^d \omega^{(j-m)k}.
\]
This is zero when $j \not\equiv m \pmod{d}$, which is equivalent to $j \neq m$ since $1 \leq j,m \leq d$.

Finally for $\psi \neq \sigma$, for every $k$, there is some $j \neq k$ such that $\sigma j \neq \psi j$.
(That is, the permutations $\psi$ and $\sigma$ have different values in at least two positions.)
So every $e_{\sigma,k}(P_{\psi,m})$ is a product involving a factor given by the $x_{j,\sigma j}$ coordinate
of $P_{\psi,m}$, but that coordinate is zero.
\end{proof}

\section{Symmetries}

In this section $\field$ is a field.

Let $V$ be the vector space over $\field$ spanned by the variables $x_{i,j}$, so $\dett_d \in \Sym^d(V)$.
Linear automorphisms of $V$ induce automorphisms of $\Sym^d(V)$.
Following \cite{Conner_2019}, it is natural to ask which of these transformations fix the decomposition \eqref{equation: explicit power sum decomposition},
that is, which linear automorphisms of $V$ leave the set $\{T_{\sigma,j}\}$ invariant.
In other words, they should leave invariant the set of linear forms whose $d$th powers occur in \eqref{equation: explicit power sum decomposition}.
However, note that said linear forms are only unique up to a $d$th root of unity, and the $d$th powers occur with coefficients of $\pm 1$, which must be preserved.

It is helpful to describe these linear forms in terms of their matrix of coefficients.
We write $E_{i,j}$ for the $d \times d$ matrix with a $1$ entry in the $(i,j)$ position
and all other entries $0$.
We write $P_\sigma$ for the permutation matrix $P_\sigma = \sum_{i=1}^d E_{i,\sigma i}$.
Note that for a $d \times 1$ column vector $(v_i)$, we have
\[
  P_\sigma \begin{pmatrix} v_1 \\ v_2 \\ \vdots \\ v_d \end{pmatrix} = \begin{pmatrix} v_{\sigma 1} \\ v_{\sigma 2} \\ \vdots \\ v_{\sigma d} \end{pmatrix}.
\]
Note also that for $\sigma,\psi \in S_d$, $P_{\sigma} P_{\psi} = P_{\sigma \psi}$.
The permutation matrices are orthogonal: $P_{\sigma}^{-1} = P_{\sigma}^t = P_{\sigma^{-1}}$.
We have $(-1)^\sigma = \det P_\sigma$.
Let $ D = \diag\{\omega,\omega^2,\dots,\omega^d\} $.
Now the linear form that appears in the term $T_{\sigma,j}$ has its coefficients given by the matrix $D^{j}P_\sigma$.
In conclusion, we are looking for linear automorphisms of $V$ which preserve the set $ \{D^jP_\sigma\} $ up to factors of $d$th roots of unity, i.e., linear automorphisms which leave the set $\{\omega^kD^{j}P_\sigma\}$ invariant.
Beyond this, some terms occur in \eqref{equation: explicit power sum decomposition} with coefficient $1$ and others with coefficient $-1$; we want our linear automorphisms to respect those coefficients.

\begin{definition}
    Let $\tilde{G}$ be the group of linear automorphisms of $V$ which leave the set of matrices $M = \{\omega^k D^d P_\sigma\}$ invariant and let $G$ be the subgroup of $\tilde{G}$ which preserves the determinant.
    We call the elements of $G$ the symmetries of the decomposition \eqref{equation: explicit power sum decomposition}.
\end{definition}
We're looking for the group $G$, which is a finite group.
We are not able to give the full group,
but we describe a subgroup of order $d^3 \varphi(d) \cdot d!/2$.

A theorem of Frobenius \cite{Frobenius} (see also \cite{MO-Frobenius}) states that all linear automorphisms $L$ of $V$ that fix $\dett_d$ (in the sense that $\det(L(X)) = \det(X)$ for all $X$) are of the form $X \mapsto AXB$ or $X \mapsto AX^t B$, 
where $X = (x_{i,j})$ and $A,B$ are $d \times d$ matrices with $\det(AB)=1$.
The transformations involving transposition are difficult to analyze,
so we consider only the transformations $AXB$.
Our question, then, is which $ A $ and $ B $ satisfy $ A(D^mP_{\rho})B \in M$.

The set $M$ is invariant under multiplication by powers of $ \omega $, left multiplication by $D$, and right multiplication by arbitrary permutation matrices.
Half of the permutations also preserve the determinant (the other half reverse the sign).
Note that $\det(D) = (-1)^{d+1}$, so depending on the parity of $d$, multiplication by $D$ may also preserve or reverse the sign of the determinant.

As for left multiplication by permutation matrices, the decomposition is preserved by certain permutations corresponding to what one might call \emph{affine linear transformations} of $\bbZ/d\bbZ$. 
Let $ \Aff_d $ be the set of permutations $\sigma$ in $S_d$
for which there exist $a,b\in[d]$ such that $\sigma i = ai+b \bmod{d}$ for all $i\in [d]$.
This is a subgroup of $S_d$ with order $d\varphi(d)$, where $\varphi$ is Euler's totient function.
In fact $\Aff_d \cong \bbZ/d\bbZ \rtimes (\bbZ/d\bbZ)^*$.
The permutation $\sigma i = i+b \bmod{d}$ has sign $(-1)^{b(d+1)}$; for $a \in (\bbZ/d\bbZ)^*$,
the sign of the permutation $\sigma i = ai \bmod{d}$ is the Jacobi symbol $\left(\frac{a}{d}\right)$ if $d$ is odd, or $(-1)^{(\frac{d}{2}+1)(\frac{a-1}{2})}$ if $d$ is even \cite{MSE-Zolotarev}.

\begin{lemma}
    Let $ \pi \in S_d $.
	Then $ P_\pi D \in M $ if and only if $ \pi \in \Aff_d $.
\end{lemma}
\begin{proof}
	Suppose that $ P_\sigma D = \omega^bD^aP_\psi $.
	Then, by comparing the locations of nonzero entries on both sides we have that $ \psi = \sigma $.
	Hence $ P_\sigma D P_{\sigma^{-1}} = \omega^bD^a$.
	The two sides are equivalent to $ \diag\{\omega^{\sigma 1},\omega^{\sigma 2},\dots, \omega^{\sigma d}\} $ and $ \diag\{\omega^{a+b},\omega^{2a+b},\dots,\omega^{da+b}\} $.
	Hence $ \omega \in \Aff_d $.
	
	Conversely, if $ \sigma\in \Aff_d $ so that $ \sigma i = ai+b \pmod d $ then $ P_\sigma D = \omega^bD^aP_{\sigma}$, and hence $ P_\sigma D \in M $.
\end{proof}

This proves that the map $X\mapsto \omega^m D^n P_\pi X P_\sigma$ is an element of $\tilde{G}$.
Let $\tilde{H} \subseteq \tilde{G}$ be the subgroup of elements of the form $X \mapsto \omega^m D^n P_\pi X P_\sigma$, $\pi \in \Aff_d$, $\sigma \in S_d$.
The proof of the above lemma shows that this is a subgroup, and in fact that it is a semidirect product:
\[
  \tilde{H} \cong ((\bbZ/d\bbZ \times \bbZ/d\bbZ) \rtimes \Aff_d) \times S_d.
\]
Here the $\bbZ/d\bbZ \times \bbZ/d\bbZ$ factor corresponds to multiplication by $\omega$ and $D$.

An element of $\tilde{H}$ given by $X \mapsto \omega^m D^n P_\pi X P_\sigma$ preserves the determinant,
meaning that $\det(\omega^m D^n P_\pi X P_\sigma) = \det(X)$, if and only if $\det(\omega^m D^n P_\pi P_\sigma) = 1$.
For half the elements of $H$ this holds, and for the other half, this determinant is $-1$.
Let $H \subset \tilde{H}$ be the index $2$ subgroup that preserves the determinant.
That is, $H = \tilde{H} \cap G$.

\begin{theorem}\label{thrm: symmetries of decomposition}
 Let $m,n \in [d]$, $\pi \in \Aff_d$, and $\sigma \in S_d$
 such that $\det(D)^n (-1)^\pi (-1)^\sigma = 1$.
 Then $L: X \mapsto \omega^m D^n P_\pi X P_\sigma$
 is a symmetry of the decomposition and is an element of $G$.
 \hfill $\qed$ % or: $\square$
\end{theorem}

This gives a subgroup of $G$ with order
$d^3\varphi(d)d!/2$, as claimed.
Table~\ref{table: symmetries} shows the number of symmetries
given by Theorem~\ref{thrm: symmetries of decomposition}.
However, this is not the full group of symmetries,
as in the $2 \times 2$ and $3 \times 3$ cases one can check that
transposition $X \mapsto X^t$ is another symmetry of the decomposition.
But for $d > 3$ it is not: for example $D P_{(1 \, 2)} \in M$
but $(D P_{(1 \, 2)})^t = P_{(1 \, 2)} D \notin M$.

%\begin{table}
%	\begin{tabular}{c|ccccc}
%		$ d $ & 2 & 3 & 4 & 5 & 6
%		\\
%		\hline
%		$ d^3\varphi(d)d!/2 $ & 16 & 324 & 3072 & 75000 & 311040
%	\end{tabular}
%	\caption{}
%	\label{table: symmetries}
%\end{table}

\begin{table}[ht]
	\[
	\begin{array}{c rrrrr}
%	\toprule
	d & 2 & 3 & 4 & 5 & 6 \\
	\midrule
	d^3\varphi(d)d!/2 & 8 & 162 & 1536 & 37500 & 15552 
%	\\
%	\bottomrule
	\end{array}
	\]
	\caption{Number of symmetries in the subgroup $H$.}
	\label{table: symmetries}
\end{table}

It would be interesting to fully characterize the whole
group of symmetries $G$, where $A,B$ are not necessarily given by
powers of $D$ or permutation matrices.
Additionally, 
it would be interesting to characterize the symmetries of the decomposition
of the form $X \mapsto AX^t B$.

\begin{remark}
The decomposition \eqref{equation: explicit power sum decomposition} is certainly not unique. For example, if $A,B$ are any two $d \times d$ matrices with $\det(AB)=1$ then one obtains a decomposition of $\det(X) = \det(AXB)$ as a sum of $d$th powers of linear forms with coefficient matrices given by $A D^j P_{\sigma} B$. Transposition yields more decompositions.
\end{remark}

\section{Defining equations}

We describe defining equations of the set of matrices $M$, up to scalar multiple.
That is, we give set-theoretic defining equations for the set of projective points $\{[D^j P_\sigma]\}$.
The equations we describe are quadrics of two types:
first, monomials, products of two distinct matrix entries from the same row or column,
and second, certain quadrics in row-sums.

In this theorem, $x_{i,j}$ are coordinates on the space of $d \times d$ matrices.
\begin{theorem}
For each $i$, let $\rho_i = x_{i,1} + \dotsb + x_{i,d}$, the $i$th row-sum.
Let $I$ be the ideal generated by the quadrics $x_{i,j_1} x_{i,j_2}$ for all $i$ and all $j_1 \neq j_2$, $x_{i_1,j} x_{i_2,j}$ for all $j$ and all $i_1 \neq i_2$, and $\rho_i^2 - \rho_{i-1} \rho_{i+1}$ for all $i$, with indices considered modulo $d$.
The common zero locus of these equations is exactly the set of projective points $\{[D^j P_\sigma]\}$.
\end{theorem}
\begin{proof}
The monomial generators cut out the set of matrices with at most one nonzero entry in each row and column.
Such matrices have the form $\Delta P_\sigma$ for some diagonal $\Delta$ and permutation $\sigma$, unique as long as $\Delta$ is nonsingular.
On this set, the values of $\rho$ are the entries of $\Delta$.
The equations $\rho_i^2 - \rho_{i-1}\rho_{i+1}$ for $2 \leq i \leq d-1$
ensure that the entries of $\Delta$ form a geometric series $(s^{d-1},s^{d-2}t,\dotsc,t^{d-1})$ for some $s,t$
(they are the familiar equations of the rational normal curve in $\bbP^d$ parametrized by $[s^{d-1} : s^{d-2} t : \dotsb : t^{d-1}]$).
The other equations at $i=1,d$ ensure that $s^d=t^d$, i.e., $t/s$ is a $d$th root of unity.
Then, up to a scalar multiple, the entries of $\Delta$ are $(w^j,w^{2j},\dotsc,w^{dj})$ for some $j$, i.e., $\Delta = D^j$.
\end{proof}
These are certainly not the generators of the full ideal of this set of points.
For $d=3$, additional generators are given as follows: $x_{i,j}^2 - P_{i;j}$, where $P_{i;j}$ is the permanent of the $2 \times 2$ submatrix complementary to the entry $x_{i,j}$.
(Our equations $\rho_i^2 - \rho_{i-1} \rho_{i+1}$ are sums of those generators, up to some monomials.)
For $d=4$, some additional generators are given by $x_{i,j_1}^2 + x_{i,j_2}^2 - P_{i,i+2;j_1,j_2}$, where $P_{i,i+2;j_1,j_2}$ is the permanent of the $2 \times 2$ submatrix complementary to rows $i,i+2$ (i.e., having rows $i-1$ and $i+1$) and to columns $j_1,j_2$.
More generators are given by $P_{i_1,i_2;j_1,j_2} - P_{i_3,i_4;j_3,j_4}$ where $\{i_1,\dotsc,i_4\}=\{j_1,\dotsc,j_4\}=\{1,\dotsc,4\}$ and $i_1+i_2 \equiv i_3+i_4 \pmod{4}$.

It would be interesting to describe these ideals in general.

\bigskip

\section*{Acknowledgements}

This work was supported by a grant from the Simons Foundation (\#354574, Zach Teitler).
We thank Gianni Krakoff and Kayla Krakoff for contributions in early conversations about this project,
the authors of \cite{Conner:2019aa} for sharing with us their work,
Jaros{\l}aw Buczy{\'n}ski and Fulvio Gesmundo for a number of very helpful suggestions,
and J.M.\ Landsberg for sharing with us Gurvits's result,
as well as encouragement.

\bigskip

\bibliographystyle{amsplain}       % Set the bibliography style to AMS
                                % alphabetized. (Can use ``amsalpha'' or
                                % ``abbrv''instead.) [amsplain, plain]
\renewcommand{\MR}[1]{}
\providecommand{\bysame}{\leavevmode\hbox to3em{\hrulefill}\thinspace}
\providecommand{\MR}{\relax\ifhmode\unskip\space\fi MR }
% \MRhref is called by the amsart/book/proc definition of \MR.
\providecommand{\MRhref}[2]{%
  \href{http://www.ams.org/mathscinet-getitem?mr=#1}{#2}
}
\providecommand{\href}[2]{#2}

\bigskip

\end{document}